\documentclass[reqno]{amsart}
\usepackage{amsmath}
\usepackage{amssymb}
\usepackage{mathabx}
\usepackage{paralist}
\usepackage[colorlinks=true]{hyperref}
\hypersetup{urlcolor=blue, citecolor=red}
\usepackage[title]{appendix}

\theoremstyle{plain}
\newtheorem{theorem}{Theorem}[section]
\newtheorem{proposition}[theorem]{Proposition}
\newtheorem{lemma}[theorem]{Lemma}
\newtheorem{corollary}[theorem]{Corollary}
\theoremstyle{definition}

\theoremstyle{remark}

\numberwithin{equation}{section}
\numberwithin{theorem}{section}
\numberwithin{figure}{section}

\newcommand{\mc}[1]{{\mathcal #1}}

\newcommand{\bb}[1]{{\mathbb #1}}

\newcommand{\rme}{\mathrm{e}}

\newcommand{\rmd}{\mathrm{d}}
\newcommand{\eps}{\varepsilon}

\newcommand{\supp}{\mathop{\rm supp}\nolimits}

\title[Time evolution of concentrated vortex rings with large radius]{Long time evolution of concentrated vortex rings with large radius}

\author[P.\ Butt\`a]{Paolo Butt\`a} 
\address{Dipartimento di Matematica\\
Sapienza Universit\`a di Roma\\
P.le Aldo Moro 5, 00185 Roma\\
Italy}

\email{butta@mat.uniroma1.it}

\author[G.\ Cavallaro]{Guido Cavallaro} 
\address{Dipartimento di Matematica\\
Sapienza Universit\`a di Roma\\
P.le Aldo Moro 5, 00185 Roma\\
Italy}

\email{cavallar@mat.uniroma1.it}

\author[C.\ Marchioro]{Carlo Marchioro}
\address{Dipartimento di Matematica\\
Sapienza Universit\`a di Roma\\
P.le Aldo Moro 5, 00185 Roma\\
Italy\\
and\\
International Research Center M\&MOCS\\ 
Universit\`a di L'Aquila\\
Palazzo Caetani\\
04012 Cisterna di Latina (LT)\\
Italy}

\email{marchior@mat.uniroma1.it}

\date{}

\begin{document}

\begin{abstract} 
We study the time evolution of an incompressible fluid with axial symmetry without swirl when the vorticity is sharply concentrated on $N$ annuli of radii of the order of $r_0$ and thickness $\eps$. We prove that when $r_0= |\log \eps|^\alpha$, $\alpha>1$, the vorticity field of the fluid converges for $\eps \to 0$ to the point vortex model, in an interval of time which diverges as $\log|\log\eps|$. This result generalizes a previous paper  \cite{CavMarJMP}, that assumed $\alpha>2$ and in which the convergence was proved for short times only.
\end{abstract}

\keywords{Incompressible Euler flow, vortex rings, point vortex model, axial symmetry}

\subjclass[2020]{76B47; 37N10}

\maketitle

\thispagestyle{empty}

\section{Introduction} \label{sec1}

In the present paper, we study the motion of an incompressible inviscid fluid with axial symmetry without swirl when the initial vorticity is very concentrated on $N$ annuli of radii of the order of $r_0$ and thickness $\eps$. In a suitable joined limit $r_0\to \infty$,  $\eps \to 0$,  we prove the convergence of this motion to the so-called point vortex system, see, e.g., \cite{mpulv}.

The motion of an incompressible inviscid fluid is described by the Euler equations, which in case of density equal to one, in three dimensions, are given by
\begin{gather}
\label{1.1}
(\partial_t + (u\cdot \nabla))\omega=(\omega\cdot \nabla )u\,, 
\\ \label{1.2}
\nabla \cdot u =0 \,,
\end{gather}
where $\omega = \nabla \wedge u$ is the vorticity, accompanied by the initial condition for the velocity field $u(x,0)=u^{0}(x)$, and the boundary condition. We suppose that the velocity vanishes as $|x| \to \infty$, which allows to reconstruct the velocity from the vorticity,
\begin{equation}
\label{1.3}
u(x,t)= - \frac{1}{4\pi} \int\!\rmd y\, \frac{x-y }{|x-y|^3} \wedge \omega(y,t)\,.
\end{equation}
We use now cylindrical coordinates $(z,r,\theta)$ and restrict our attention to initial data with axial symmetry without swirl, i.e., the component $u^0_\theta$ is zero and $u^0_z, u^0_r$ do not depend on $\theta$. Since the time evolution preserves this symmetry, the corresponding solutions have the same form,
\begin{equation}
\label{1.4}
u(x,t) = (u_z,u_r,u_\theta) = (u_z(z,r,t), u_r(z,r,t), 0)\,.
\end{equation}
In this case, the  vorticity reduces to a scalar,
\begin{equation}
\label{1.5}
\omega = (0,0,\omega_{\theta}) = (0,0, \partial_zu_r - \partial_r u_z)\,,
\end{equation}
and Eqs.~\eqref{1.1}-\eqref{1.2} become
\begin{gather}
\label{1.6}
\partial_t \omega_\theta  + (u_z\partial_z + u_r\partial_r) \omega_\theta - \frac{u_r \omega_\theta}{r} = 0 \,,
\\ \label{1.7}
\partial_z(r \ u_z)+\partial_r(r \ u_r) = 0\,.
\end{gather}
From now on, we denote $\omega_{\theta}$ by $\omega$. Finally, Eq.~\eqref{1.3} reads,
\begin{align}
\label{1.8}
u_z & = - \frac{1}{2\pi} \int\! \rmd z' \!\int_0^\infty\! r' \rmd r' \! \int_0^\pi\!\rmd \theta \, \frac{\omega(z',r',t) (r\cos\theta - r')}{[(z-z')^2 + (r-r')^2 + 2rr'(1-\cos\theta)]^{3/2}} \,,
\\ \label{1.9}
u_r & = \frac{1}{2\pi} \int\! \rmd z' \!\int_0^\infty\! r' \rmd r' \! \int_0^\pi\!\rmd \theta \, \frac{\omega(z',r',t) (z - z') \cos\theta}{[(z-z')^2 + (r-r')^2 + 2rr'(1-\cos\theta)]^{3/2}} \,.
\end{align}
Hence, the axially symmetric solutions to the Euler equations are given by the solutions to Eqs.~\eqref{1.5}-\eqref{1.9}. Eq.~\eqref{1.6} means that the quantity $\omega/r$ remains constant along the flow generated by the velocity field, i.e.,
\begin{equation}
\label{1.10}
\frac{\omega(z(t),r(t),t)}{r(t)}= \frac{\omega(z(0),r(0),t)}{r(0)} \, ,
\end{equation}
where $(z(t),r(t))$ solve
\begin{equation}
\label{1.11}
\dot{z}(t)=u_z(z(t),r(t),t) \, ,  \qquad  \dot{r}(t)=u_r(z(t),r(t),t) \, .
\end{equation}
It is possible to introduce an equivalent weak formulation of Eq.~\eqref{1.6} that allows to consider non-smooth initial data; by a formal integration by parts we obtain indeed
\begin{equation}
\label{1.12}
\frac{\rmd}{\rmd t} \omega_t[f,t]=\omega_t \big[u_z\partial_z f + u_r\partial_r f + \partial_tf \big] \, ,
\end{equation}
where $f=f(z,r,t)$ is a bounded smooth test function and
\begin{equation}
\label{1.13}
\omega_t [f,t] :=\int\!\rmd z \int_0^\infty\! \rmd r\, \omega (z,r,t) f(z,r,t)\,.
\end{equation}
The existence of global axially symmetric solutions both for the Euler and Navier-Stokes equations has been established many years ago \cite{Lad68,UkY68}, see also \cite{Fe-Sv,Gal13,Gal12} for more recent results.
Global  in time  existence and uniqueness of a weak solution to the associated Cauchy problem holds when the initial vorticity is a bounded function with compact support contained in the open half-plane  $\Pi:=\{(z,r)\colon  r>0\}$, see for instance \cite[pag.~91]{mpulv} or \cite[Appendix]{CS}. In particular, it can be shown that the support of the vorticity remains in the open half-plane $\Pi$ at any time. A point in the half-plane $\Pi$ corresponds to a circumference in the whole space. The special class of axisymmetric solutions without swirl are called sometimes \textit{vortex rings}, because there exist particular solutions whose shape remains constant in time (the so-called steady vortex ring) and translate in the $z$-direction with constant speed (see for instance \cite{Fra70}). The existence and the properties of these solutions is an old question. For a rigorous proof by means of variational methods see \cite{AmS89,FrB74}. For references on axially symmetric solution without swirl see also the review paper \cite{ShL92}.

As said before, we want to describe the evolution of the vorticity when it is initially concentrated in $N$ annuli of radii of the order of $r_0$ and thickness $\eps \in (0,1)$, in a joined limit $\eps\to 0$, $r_0 \to \infty$. This asymptotic regime can be appropriately formulated in the system of coordinates
\begin{equation} 
\label{coordinate} 
x=(x_1, x_2) := (z, r-r_0)\,.
\end{equation}
More precisely, we assume that initially the vorticity is concentrated in $N$ ``small blobs'', i.e., it can be expressed as
\begin{equation}
\omega_\eps(x,0)= \sum_{i=1}^N \omega_{i,\eps}(x,0)\,,
\label{in_data}
\end{equation}
where each $\omega_{i,\eps}(x,0)$  is a non-negative or non-positive function such that, denoting by $\Sigma(\xi| \rho)$ the open disk of center $\xi$ and radius $\rho$ in $\mathbb{R}^2$,
\begin{equation}
\label{in_data2}
\Lambda_{i,\eps}(0):= \supp\, \omega_{i,\eps}(\cdot,0) \subseteq \Sigma(z_i | \eps)\,, \quad i=1,\ldots,N\,,
\end{equation}
being $\eps>0$ a small parameter and $\{z_i\}_{i=1, \dots, N}$ points in $\bb R^2$ fixed independently of $\eps$ and $r_0$. As we are interested in a jointed limit $\eps\to 0$ and $r_0\to \infty$, we can suppose without loss of generality that
\[
\Sigma(z_i | \eps) \subset \Pi =\{x \colon  r_0 + x_2> 0\}\,, \qquad  \Sigma(z_i |\eps) \cap \Sigma(z_j |\eps) =\emptyset \quad \forall i \ne j\,.
\]
We finally assume that, for any $i=1, \dots, N$,
\begin{equation}
a_i := \int\!  \rmd x\, \omega_{i,\eps}(x,0) \in \bb R
\end{equation}
is independent of $\eps$, and
\begin{equation}
\label{mass}
|\omega_{i,\eps}(x,0)| \le M \eps^{- 2} \; , \qquad M>0  \, .
\end{equation}

We are going to investigate in which situations the time evolution of these states almost preserves its form. In Theorem \ref{teoEA} below, we will prove that if the assumption $r_0=|\log\eps|^\alpha$ with $\alpha > 1$ is fulfilled then the evolved state $\omega_\eps(x,t)$ can be written as
\begin{equation}
\label{t_data}
\omega_\eps(x,t) = \sum_{i=1}^N \omega_{i,\eps}(x,t)\,,
\end{equation}
where each $\omega_{i,\eps}(x,t)$ is a non-negative or non-positive function such that
\[
\Lambda_{i,\eps}(t):= \text{supp} \, \omega_{i,\eps}(\cdot, t) \subset \Sigma(z_i(t) | r_t(\eps))\,,
\]
with
\[
\Sigma(z_i(t) | r_t(\eps)) \cap  \Sigma(z_j(t) | r_t(\eps)) = \emptyset   \qquad \forall\, i \ne j\,,
\]
being $r_t(\eps)$ a positive function, vanishing for $\eps\rightarrow 0$, and
$\{z_i(t)\}_{i=1,\ldots,N} \subset \mathbb{R}^2$ solution to the point-vortex model, i.e., the dynamical system defined by
\begin{equation}
\label{pvmodel}
\dot{z}_i(t) = \sum_{j \ne i} a_j K(z_i(t)- z_j(t))\,, \quad z_i(0)=z_i\,, \qquad i=1, \ldots, N\,,
\end{equation}
where
\begin{equation}
\label{nucleoK}
K(x)=-\frac{1}{2\pi}\nabla^{\perp}\log|x|, \qquad \nabla^{\perp}=(\partial_2,-\partial_1),
\end{equation}
with $-\frac1{2\pi} \log|x|$ the fundamental solution of the Laplace operator in $\mathbb{R}^2$. When all the $a_i$ have the same sign there is a global solution, otherwise there is a finite time at which a collapse (that is, a pair of $z_i$ arriving at the same point) or a $z_i$ going to infinity can occur. However, these events are exceptional, see for instance \cite{mpulv}.

This dynamical system has been introduced a long time ago by Helmholtz, as a particular \textit{solution} to the Euler equations \cite{Hel67}, and investigated by many authors \cite{Kel,Kir,Poi}. It can be used to describe the time evolution of irregular initial data and it is the basis of an approximation method (the so-called vortex method) in which $N \to \infty$ and $a_i\to 0$ (for more information, see for instance the textbooks \cite{mpulv}, \cite{MaB02}, and the references in \cite{CGP14}).

Even if solutions to Eq.~\eqref{pvmodel} are not solution of the Euler equations, they can be an average of different solutions that in $\bb R^3$ are clusters of straight lines of vorticity, as it is discussed in \cite{CapMar,Mar88, Mar, MaP,MaP93,Turk}. In \cite{marc}, it is shown that the same happens when the straight lines are changed into large enough annuli, with radius of the order $r_0(\eps)= \eps^{-\alpha}$, for any $\alpha>0$ and for any finite time, result which has been extended in \cite{CS} to long times. 

In the present paper, we assume a weaker dependence of $r_0$ on $\eps$, precisely we will show that the relation with the point vortex model remains valid for long times, when considering radii of the order $|\log\eps|^\alpha$, $\alpha>1$, thus improving what obtained in \cite{CavMarJMP}, in which it is assumed $\alpha>2$ and the convergence was proved for short times only. This is achieved by means of a more refined application of the iterative method with respect to \cite{CavMarJMP}, borrowing some ideas recently introduced in \cite{butcavmar} to study the case $\alpha=1$. In particular, the splitting into two different iterative methods and the extension of the convergence time by repeated applications of the former iterative methods.

The case $\alpha=1$ describes a critical regime, in which the self-induced field acting on each ring (responsible of the uniform drift along the symmetry axis) and the interaction with the other rings are of the same order. In this case, as conjectured in \cite{marneg} and recently proved in \cite{butcavmar}, the rings still preserve their shape and move according to the dynamical system defined by
\begin{equation}
\label{dyn2}
\dot{z}_i(t)= \sum_{j \ne i} a_j K(z_i(t)- z_j(t)) + a_i \begin{pmatrix} 1 \\ 0 \end{pmatrix} \,, \quad z_i(0) = z_i \,, \qquad i=1, \ldots, N\,.
\end{equation}
Concerning the content of \cite{butcavmar}, we notice that Eq.~\eqref{dyn2} admit periodic solutions, which correspond to leapfrogging dynamics of the vortex rings. Moreover, beside the convergence is proved for short times only, it is possible to chose initial data for which the convergence to a periodic solution cover several periods. This provide a mathematical justification of the leapfrogging, which was first proved in \cite{DDMW} (but with different techniques and scaling).

For $\alpha=0$ (see \cite{ben} for a vortex ring alone and \cite{butmar2, BCM} for several rings) the dynamics converges to simple translations parallel to the symmetry axis with constant speed, while the case $0<\alpha<1$  has not been studied explicitly, but we believe that the behavior is analogous to the case $\alpha=0$.

We finally mention that also the case of viscid fluids is discussed in the literature (mainly in the limit of small viscosity) \cite{BCM2,BrM11,CS,Gal11,Gal14,Lad68, Mar90,marcNS,Mar07,UkY68}, but this topic is out of the scope of the present analysis.

\section{Main result} 
\label{sec:2}

{\textit{A warning on the notation}}. Hereafter in the paper, we denote by $C$
a generic positive constant (eventually changing from line to line) which is independent of the parameter $\eps$ and the time $t$.

We fix $\alpha>1$, set $r_0 = r_0(\eps) = |\log\eps|^\alpha$, and consider an initial vorticity as specified in Eqs.~\eqref{in_data}-\eqref{mass}. Given the velocity field $u(x,t)$, we define the trajectory of a fluid element starting at $x$ as the solution of the integral equation
\[
\phi_t(x) = x + \int_0^t\!\rmd s\, u(\phi_s(x),s)\,,
\]
with components $\phi_t(x)= \left( \phi_{t,1}(x),  \phi_{t, 2}(x) \right)$. Since the quantity $\omega/r$ remains constant along the flow generated by the velocity field, see Eq.~\eqref{1.10}, the evolution at time $t$ can be decomposed as in Eq.~\eqref{t_data} provided
\begin{equation}
\label{term i EA}
\omega_{i, \eps}(x,t) :=  \frac{r_0 + x_2}{r_0+\phi_{-t, 2}(x)} \omega_{i, \eps}(\phi_{-t}(x),0)\,.
\end{equation}
In particular, the support of $\omega_{i,\eps}(\cdot, t)$ is given by $\Lambda_{i,\eps}(t) = \{\phi_t(x)\colon x\in \Lambda_{i,\eps}(0)\}$. Since $\Lambda_{i,\eps}(t) \subset \{x\colon x_2> -r_0\}$, hereafter we extend $\omega_{i, \eps}(x,t)$ to a function defined on the whole plane $\bb R^2$ by setting $\omega_{i, \eps}(x,t) = 0$ if $x_2 \le - r_0$.

We remark that each $\omega_{i,\eps}(x,t)$ preserves the initial sign (this is obvious), whence, by the assumptions on the initial datum, it is a non-negative or non-positive function. Moreover, it preserves the mass, i.e.,
\begin{equation}
\label{norme}
\int\!\rmd x\, \omega_{i, \eps}(x,t) = a_i \quad \forall\, t\ge 0\,.
\end{equation}
Indeed, Eq.~\eqref{norme} is a direct consequence of the conservation of $\omega/r$, after integrating in $\bb R^3$, adopting cylindrical coordinates, and using that the volume is conserved by the Euler flow.

Let us call $\{z_i(t) \}_{i=1, \ldots, N}$ the solution to the point vortex system Eq.~\eqref{pvmodel} with intensities $a_i$ and initial data the point $\{z_i\}_{i=1,\ldots,N}$ appearing in Eq.~\eqref{in_data2}. As already remarked, such solution is defined globally in time apart from a zero measure set of initial data. Suppose now that $\{z_i\}_{i=1,\ldots,N}$ are chosen such that the solution to Eq.~\eqref{pvmodel} is global. As discussed before, under appropriate assumptions on the radius $r_0$, we expect the fluid motion converges to the point vortex system for fixed positive time. On the other hand, as time goes by, small filaments of fluid could move very far away, so we cannot expect that convergence takes place uniformly in time. Since in a realistic situation the parameter $\eps$ is not zero, a natural question is to characterize the larger time interval on which this approximation remains valid, for small but positive values of the parameter $\eps$. 

Our analysis show that this maximal time interval can be indeed quite long, i.e., diverging as $\eps\to 0$. More precisely, we prove that this is the case under the assumption that the initial configuration $\{z_i\}_{i=1,\ldots,N}$, in addition to the request of existence of global solution $\{z_i(t)\}_{i=1,\ldots,N}$, satisfies the stronger property that
\begin{equation}
\label{R_m}
R_m := \inf_{t \ge 0} \min_{i \neq j}  |z_i(t) -z_j(t)| >0 \,.
\end{equation}

To state precisely the result we need some preliminary notation. We fix
\begin{equation}
\label{betabeta}
0 < \beta < \frac{\alpha-1}{4}
\end{equation}
and define
\begin{equation}
\label{T_eps_beta}
T_{\eps,\beta} := \min_i \sup\left\{t>0 \colon |x-z_i(s) |\le |\log\eps|^{-\beta} \;\; \forall x\in \Lambda_{i,\eps}(s) \;\; \forall s\in[0, t]  \right\}\,.
\end{equation}
Otherwise stated, $T_{\eps, \beta}$ is the first time (if any) in which a filament reaches the boundary of  $\bigcup_i \Sigma(z_i(t) |\, |\log\eps|^{-\beta})$. We observe that, eventually as $\eps \to 0$, by continuity $T_{\eps,\beta}>0$ in view of Eq.~\eqref{in_data2} and, by Eq.~\eqref{R_m},
\begin{equation}
\label{Rm/2}
\mathrm{dist}\big(\Sigma(z_i(t) |\, |\log\eps|^{-\beta}),\Sigma(z_j(t) |\, |\log\eps|^{-\beta}) \big) \ge \frac{R_m}2 \quad \forall\,i\ne j\, \;\; \forall\, t\ge 0\,,
\end{equation}
which implies
\begin{equation}
\label{Rm3}
\mathrm{dist}\big(\Lambda_{i,\eps}(t),\Lambda_{j,\eps}(t) \big) \ge \frac{R_m}2 \quad \forall\,i\ne j\, \;\; \forall\, t \in [0,T_{\eps,\beta}]\,.
\end{equation}
Hereafter, we shall always assume $\eps$ so small that $T_{\eps,\beta}>0$ and Eq.~\eqref{Rm/2} hold true.

Clearly, $T_{\eps, \beta}$ gives a lower bound of the time horizon where the point vortex approximation is valid. We can now state the main result of the paper, which is achieved by adapting the strategies given in \cite{butmar,CavMarJMP}, with some new ideas introduced in \cite{butcavmar}.

\begin{theorem} 
\label{teoEA}
Let $r_0 = |\log\eps|^{\alpha}$ with $\alpha>1$ and assume that the initial vorticity satisfies Eqs.~\eqref{in_data}-\eqref{mass} and Eq.~\eqref{R_m}. Then, for each $\beta\in \big(0, \frac{\alpha-1}4\big)$ there exist $\eps_0>0$ and $\zeta >0$ such that 
\[
T_{\eps, \beta} \ge \zeta \log|\log\eps| \qquad \forall \eps\in (0, \eps_0)\,.
\]
\end{theorem}

Concerning the assumption Eq.~\eqref{R_m}, we remark that $R_m>0$ is a well known property of the dynamics for vortices with intensities of the same sign. In the general case, the existence of a unique global solution to the Cauchy problem to Eq.~\eqref{pvmodel} is proved for any choice of initial data and intensities  $\{z_i, a_i\}$, $i=1,\dots, N$, outside a set of Lebesgue measure zero \cite{mpulv}. This fact does not implies $R_m>0$, but it makes this assumption very reasonable.

\section{Proof of Theorem \ref{teoEA}}
\label{sec:3}

The proof of Theorem \ref{teoEA} is rather technical and composed of many auxiliary lemmata and propositions. The general strategy can be summarized as follows. We analyze the motion of a tagged vortex (ring) with index $i$ during the time interval $[0,T_{\eps,\beta}]$ under the influence of the others. For any $\eps$ small enough, the field generated by the remaining $N-1$ vortices has the features of a given external bounded field, since the minimum distance between any two distinct vortices remains greater than a positive constant, see Eq.~\eqref{Rm3}. Then, by making use of a basic estimate on the growth in time of the moment of inertia of the tagged vortex, we estimate the vorticity mass far from the center of vorticity, showing that it is negligible when $\eps$ is small (to this purpose, we adopt the aforementioned iterative method). This in turn  implies a control on the maximal diameter of the support of the tagged vortex, which allows to derive the desired lower bound on $T_{\eps, \beta}$. Some preliminary results are quite straightforward adaptations to the present case of the ones in previous papers \cite{CavMarJMP, butmar,butmar2,CapMar,CS,ISG,marc}. Therefore, for some of these results, we shall only sketch the proofs.

Hereafter, we assume that the initial distribution $\omega_\eps(x,0)$ satisfies Eqs.~\eqref{in_data}-\eqref{mass}. In particular, $\omega_\eps(x,t)$ can be decomposed as in Eq.~\eqref{t_data} with $\omega_{i,\eps}(x,t)$ as in Eq.~\eqref{term i EA}.

Making use of Eqs.~\eqref{1.8}-\eqref{1.9}, written with respect to the new coordinate system Eq.~\eqref{coordinate}, the velocity field reads,
\[
u(x,t) = \int\!\rmd y\,  G(x,y) \, \omega_\eps(y, t)\,,
\]
where $G(x,y)$ denotes the integral kernel appearing in Eqs.~\eqref{1.8}-\eqref{1.9} expressed in the new coordinates Eq.~\eqref{coordinate}. Furthermore, for each index $i$, it can be decomposed in the form
\begin{equation}
\label{u=ui+F}
u(x,t) =u^i(x,t) + F_\eps^i(x,t)\,,
\end{equation}
where
\begin{equation}
\label{ui}
u^i(x,t) = \int\! \rmd y\, G(x,y)\,\omega_{i,\eps}(y,t)
\end{equation}
is the velocity field generated by the $i$-th vortex, and 
\begin{equation}
\label{effe_iep}
F_\eps^i(x,t) = \sum\limits_{j\neq i}\int\! \rmd y\,G(x,y)\,\omega_{j,\eps}(y,t)
\end{equation}
is the one generated by the remaining $N-1$ vortices.

We start with the analysis of the kernel $G(x,y)=(G_1(x,y),G_2(x,y))$, whose explicit expression is
\begin{equation}
\label{G1}
G_1(x,y) =  \frac{1}{2\pi} \int_0^\pi\! \rmd\theta\, \frac{ (r_0+y_2) \big[(r_0+y_2) - (r_0+x_2) \cos \theta \, \big]}{\left\{ |x-y|^2 +2(r_0+x_2)(r_0+y_2)(1-\cos \theta) \right\}^{3/2}}\,,
\end{equation}
\begin{equation}
\label{G2}
G_2(x,y) =  \frac{1}{2\pi} \int_0^\pi\! \rmd\theta\, \frac{(r_0+ y_2)(x_1-y_1) \cos \theta}{ \left\{|x-y|^2 +2(r_0+x_2)(r_0+y_2)(1-\cos \theta) \right\}^{3/2}}
\end{equation}
(which depends on $\eps$ throughout the parameter $r_0=|\log\eps|^\alpha$).

\begin{proposition}
\label{stimaD}
Consider $x,y$ such that
\begin{equation}
\label{x2y2}
|x_2| \le \frac{r_0}{2}\,, \qquad |y_2| \le \frac{r_0}{2}\,,
\end{equation}
and let $r_0 = |\log\eps|^{\alpha}$. Then, for any $\eps$ small enough,
\begin{equation}
\label{prop3.2}
|G(x,y)- K(x-y)| \le \frac{C}{|\log\eps|^\alpha} \left( 1 + \log|\log\eps| + \chi_{(0,1)}(|x-y|) \log |x-y|^{-1}\right),
\end{equation}
where $\chi_{(0,1)}(\cdot)$ denotes the characteristic function of the interval $(0,1)$.
\end{proposition}

\begin{proof}
This is the analogous of \cite[Proposition 3.2]{CS}, where the difference with our case is the choice $r_0=\eps^{-\alpha}$ instead of $r_0 = |\log\eps|^{\alpha}$. The proof given in \cite{CS} applies here straightforwardly, since $|G-K|$ is estimated in terms of $r_0$ (inserting $r_0 = \eps^{-\alpha}$ only at the end). More precisely, letting $D = G-K$, it is shown that, under the assumption Eq.~\eqref{x2y2},
\[
|D_1(x,y)| \le \frac{C}{r_0} \left(1 + \log r_0 + \chi_{(0,1)}(|x-y|) \log |x-y|^{-1} \right), \quad |D_2(x,y)| \le \frac{C}{r_0}\,.
\]
As $|G-K| \le |D_1|+|D_2|$, Eq.~\eqref{prop3.2} now follows from the above estimates with $r_0=|\log\eps|^\alpha$.
\end{proof}

We now introduce a parameter $\eta>0$ (to be chosen conveniently later on), define
\begin{equation}
\label{Tebe}
T_{\eps,\beta,\eta} = T_{\eps,\beta} \wedge \eta\log|\log\eps|\,,
\end{equation} 
and analyze the dynamics in the time interval $[0,T_{\eps,\beta,\eta}]$ when $\eps$ is small. The goal is to obtain a sharp localization property of the vorticities, which implies $T_{\eps,\beta,\eta} = \eta\log|\log\eps|$ for a suitable choice of $\eta$ and any $\eps$ small enough. This will prove the claim of Theorem \ref{teoEA} with $\zeta$ equal to this choice of $\eta$.

\begin{lemma} 
\label{norme di omega}
For any $\eps$ small enough,
\begin{equation} 
\label{norme2}
|\omega_{i,\eps}(x, t)| \le 3 M \eps^{-2} \quad \forall\,t \in [0,T_{\eps,\beta,\eta}] \quad \forall\, i=1,\ldots, N\,,
\end{equation}
with $M$ as in Eq.~\eqref{mass}.
\end{lemma}

\begin{proof}
In view of definitions Eqs.~\eqref{T_eps_beta} and \eqref{Tebe}, for any $x\in \Lambda_{i,\eps}(t)$ and $t \in [0,T_{\eps,\beta,\eta}]$ we have $x\in \Sigma(z_i(t) |\, |\log\eps|^{-\beta})$, whence
\[
|r_0| - |z_i(t)| - |\log\eps|^{-\beta} \le |r_0+x_2| \le |r_0| + |z_i(t)| + |\log\eps|^{-\beta}\,.
\]
Moreover, by Eqs.~\eqref{pvmodel} and \eqref{R_m},
\[
|z_i(t)| \le |z_i|+ \frac{1}{2\pi}\int_0^t\!\rmd s\, \sum_{j \neq i} \frac{|a_j|}{|z_i(s)- z_j(s)|} \le |z_i|+ \frac{t}{2 \pi R_m} \sum_{j \neq i} |a_j|\,.
\]
Therefore, as $T_{\eps,\beta,\eta} \le \eta\log|\log\eps|$,
\begin{equation}
\label{stimx2}
\frac{r_0}2 \le |r_0+ x_2| \le \frac 32 r_0 \quad \forall\, x\in \Lambda_{i,\eps}(t) \quad \forall\, t \in [0,T_{\eps,\beta,\eta}]\,,
\end{equation}
for any $\eps$ small enough. Inserting the above bounds in Eq.~\eqref{term i EA} we get $|\omega_{i,\eps}(x,t)| \le 3 |\omega_{i,\eps}(x_0,0)|$ for any $x\in \Lambda_{i, \eps}(t)$ with $x_0=\phi_{-t}(x)\in \Lambda_{i, \eps}(0)$, and the lemma follows in view of Eq.~\eqref{mass}.
\end{proof}

We denote by
\begin{equation}
\label{utilde}
\widetilde{u}^i(x, t) = \int\!\rmd y\, K(x-y) \, \omega_{i,\eps}(y, t)\,,
\end{equation}
with $K$ defined in Eq.~\eqref{nucleoK}, the vector field generated by the $i$-th vortex in the planar case (say, $r_0=\infty$).

\begin{proposition} 
\label{u e utilde}
For any $\eps$ small enough,
\begin{equation} 
|u^i(x,t)-\widetilde{u}^i(x, t)| \le \frac{C}{|\log\eps|^{\alpha-1}} \quad \forall\,x \in \Sigma(z_i(t) |\, |\log\eps|^{-\beta}) \quad \forall\,t \in [0,T_{\eps,\beta,\eta}]\,.
\label{u e utilde 2}
\end{equation}
for any $i=1,\ldots, N$.
\end{proposition}

\begin{proof}
By the same reasoning leading to Eq.~\eqref{stimx2}, if $t \in [0,T_{\eps,\beta,\eta}]$ then $|x_2| \le r_0/2$ and $|y_2| \le r_0/2$ for any  $x\in \Sigma(z_i(t) |\, |\log\eps|^{-\beta})$ and $y\in \bigcup_j \Lambda_{j,\eps}(t)$. Therefore, we can apply Proposition \ref{stimaD} to bound $|u^i(x,t)-\widetilde{u}^i(x, t)|$. In particular, using the decomposition Eq.~\eqref{t_data}, if $x\in \Sigma(z_i(t) |\, |\log\eps|^{-\beta})$ (for any given $i=1,\ldots,N$),
\begin{align}
\label{uu}
|u^i(x,t)-\widetilde{u}^i(x, t)| & \le \frac{C( 1 + \log|\log\eps|)}{|\log\eps|^\alpha}\,  |a_i|  \nonumber \\ & \quad + \frac{C}{|\log\eps|^\alpha} \int_{|y-x| <1}\!\rmd y\,\log |x-y|^{-1}|\omega_{i,\eps}(y, t)| \,,
\end{align}
where we used that, since each $\omega_{i,\eps}$  is a non-negative or non-positive function, Eq.~\eqref{norme} implies also
\begin{equation}
\label{norme*}
\int\!\rmd x\, |\omega_{i, \eps}(x,t)| = |a_i| \,, \quad \forall\, t\ge 0 \quad \forall\, i=1,\ldots,N\,.
\end{equation}
The last (and worst) term can be estimated performing a symmetrical rearrangement of the vorticity around the point $x$. From Eqs.~\eqref{norme2} and \eqref{norme*} we get an upper bound by replacing $\omega_{i, \eps}$ with the function equal to the constant $3 M\eps^{-2}$ in the disk of center $x$ and radius $r$ and equal to zero outside this disk, where the radius $r$ is chosen such that the total mass of vorticity is $|a_i|$, i.e., $3 M\eps^{-2}\pi r^2  =|a_i|$. Therefore,
\begin{gather*}
\int_{|y-x| <1}\!\rmd y\,\log |x-y|^{-1}|\omega_{i,\eps}(y, t)| \le 3 M \eps^{-2} \int_{|z| \le r}\! \rmd y \, \log |z|^{-1} \\ =  - 3 \pi M \eps^{-2} r^2 \log r + \frac 32 \pi M \eps^{-2} r^2 \le C (1+|\log\eps|) \,.
\end{gather*}
Inserting this bound in Eq.~\eqref{uu}, the proposition follows.
\end{proof}

\begin{lemma} 
\label{comesonoFeps}
The field $F^i_\eps$ defined in Eq.~\eqref{effe_iep} can be decomposed as
\[
F^i_\eps(x,t)= F^i_{\eps, 1}(x,t) + F^i_{\eps, 2}(x,t)\,,
\]
where 

(i) $F^i_{\eps, 1}(x,t)$ is Lipschitz and uniformly bounded. More precisely, there exists a positive constant $C_F$ such that
\begin{gather*}
|F^i_{\eps, 1}(x,t)-F^i_{\eps, 1}(x',t)| \le C_F|x-x'|\quad \forall\,x, x' \in \Sigma(z_i(t) |\, |\log\eps|^{-\beta})\,, \\ |F^i_{\eps, 1}(x,t)| \le C_F \quad \forall\,x  \in \Sigma(z_i(t) |\, |\log\eps|^{-\beta})\,,
\end{gather*}
for any $t \in [0, T_{\eps,\beta,\eta}]$ and $\eps$ small enough.

(ii) $F^i_{\eps, 2}(x,t)$ is small. More precisely, there exists a positive constant $C_F'$ such that
\[ 
|F^i_{\eps, 2}(x,t)| \le \frac{C_F'\log|\log\eps|}{|\log \eps|^{\alpha}} \quad \forall\,x \in \Sigma(z_i(t) |\, |\log\eps|^{-\beta})\,,
\]
for any $t \in [0, T_{\eps,\beta,\eta}]$ and $\eps$ small enough.
\end{lemma}

\begin{proof}
We define
\begin{align*}
F^i_{\eps, 1}(x,t) = & \sum_{j \neq i} \int\! \rmd y \, K(x-y) \, \omega_{j, \eps}(y,t)\,, \\ F^i_{\eps, 2}(x,t) = & \sum_{j \neq i} \int\! \rmd y \, [G(x,y) - K(x-y)] \, \omega_{j, \eps}(y,t)\,.
\end{align*}
By the properties of $K$ and Eq.~\eqref{Rm/2}, it follows that $F^i_{\eps, 1}(x)$, $x\in\Sigma(z_i(t) |\, |\log\eps|^{-\beta})$, is Lipschitz and bounded, uniformly as $\eps\to 0$. Moreover, as already noticed in the proof of Proposition \ref{u e utilde}, if $t \in [0,T_{\eps,\beta,\eta}]$ then $|x_2| \le r_0/2$ and $|y_2| \le r_0/2$ for any  $x\in \Sigma(z_i(t) |\, |\log\eps|^{-\beta})$ and $y\in \bigcup_j \Lambda_{j,\eps}(t)$. Therefore, by Proposition \ref{stimaD} and Eq.~\eqref{norme*}, if $x \in \Sigma(z_i(t) |\, |\log\eps|^{-\beta})$,
\begin{align*}
|F^i_{\eps, 2}(x, t)| & \le \frac{C( 1 + \log|\log\eps|)}{|\log\eps|^\alpha} \sum_{j\ne i} |a_j| \\ & \quad + \frac{C}{|\log\eps|^\alpha} \sum_{j\ne i} \int_{|y-x| <1}\!\rmd y\,\log |x-y|^{-1}|\omega_{j,\eps}(y, t)| \\ & \le \frac{C( 1 + \log|\log\eps|)}{|\log\eps|^\alpha} \sum_{j\ne i} |a_j| + \frac{C}{|\log\eps|^\alpha} \sum_{j\ne i} |a_j| \log\Big(1+\frac 2{R_m}\Big)\,,
\end{align*}
where we used Eq.~\eqref{Rm/2}. The lemma is proved.
\end{proof}

We now recall the definitions of center of vorticity and the moment of inertia,
\begin{equation}
\label{cen_vort}
B^i_\eps(t) := \frac1{a_i} \int\! \rmd x\, x \, \omega_{i,\eps}(x,t)\,, \qquad
I^i_\eps(t) := \int\! \rmd x\, |x-B^i_\eps(t)|^2 |\omega_{i, \eps}(x,t)|\,.
\end{equation}
A basic tool in the analysis is an upper bound on the moment of inertia, which is the content of the following proposition.

\begin{proposition} 
\label{lem:Ieps Beps}
For any $\eps$ small enough and $i=1,\ldots,N$,
\begin{equation} 
\label{Iepsilon} 
I^i_\eps(t) \le C \frac{|\log\eps|^{2\eta C_F }}{|\log\eps|^{2(\alpha-1)}} \quad \forall\,t \in [0,T_{\eps,\beta,\eta}]\,.
\end{equation}
\end{proposition}

\begin{proof}
We need to estimate the derivative of $I^i_\eps(t)$. To this purpose, we observe that each vortex ring $\omega_{i,\eps}(x,t)$ is transported by the same field Eq.~\eqref{u=ui+F}, so a weak formulation like Eq.~\eqref{1.12} holds. More precisely, if $\omega_{i,\eps}[f,t] := \int\!\rmd x\, \omega_{i,\eps}(x,t) f(x,t)$ then
\begin{equation} 
\label{eqderivaterid} 
\frac{\rmd}{\rmd t}\omega_{i,\eps}[f,t] = \omega_{i,\eps}\big[ (u^i+F^i_\eps)\cdot \nabla f + \partial_t f, t\big]\,.
\end{equation}
Therefore, as $|\omega_{i,\eps}(x,t)| = \frac{a_i}{|a_i|} \omega_{i,\eps}(x,t)$,
\begin{align*}
\dot I^i_\eps(t) & = \frac{a_i}{|a_i|} \int\!\rmd x\, \omega_{i,\eps}(x, t) \, \left[ (u^i+F^i_\eps)\cdot 2(x-B^i_\eps(t)) - 2 \dot B^i_\eps(t) \cdot (x-B^i_\eps(t)) \right]\,, \\ & = 2\frac{a_i}{|a_i|} \int\!\rmd x\, \omega_{i,\eps}(x, t) \, \left[ (u^i+F^i_\eps)\cdot (x-B^i_\eps(t)) \right]\,,
\end{align*}
where in the last equality we used that $\int \rmd x\, \omega_{i,\eps}(x, t) (x-B^i_\eps(t)) = 0$ by the definition  of $B^i_\eps(t)$. For the same reason we have,
\begin{align*} 
& \left| \int\! \rmd x\, \omega_{i,\eps}(x,t) F^i_\eps(x,t) \cdot (x-B^i_\eps(t)) \right| \\ & =  \left| \int\! \rmd x\, \omega_{i,\eps}(x,t) \left[ F^i_{\eps, 1}(x,t)- F^i_{\eps, 1} (B^i_\eps(t),t) + F^i_{\eps, 2}(x, t) \right]  \cdot (x-B^i_\eps(t)) \right| \\ & \le C_F\int\! \rmd x\, |\omega_{i,\eps}(x,t)| \, |x-B^i_\eps(t)|^2 +  \frac{C_F'\log|\log\eps|}{|\log\eps|^{\alpha}} \int\! \rmd x\, \omega_{i,\eps}(x,t) |x-B^i_\eps(t)| \\ & \le C_F \, I^i_\eps(t) + \frac{C_F'\log|\log\eps|}{|\log\eps|^{\alpha}}I^i_\eps(t)^{1/2}\,,
\end{align*}
where we used the Cauchy-Schwarz inequality in the last line. Note that the Lipschitz property of $F^i_{\eps,1}$, stated in Lemma \ref{comesonoFeps}, can be applied here since, by convexity, $B^i_\eps(t) \in  \Sigma(z_i(t) |\, |\log\eps|^{-\beta})$ regardless of whether it belongs to $\Lambda_{i,\eps}(t)$ or not. For  the term containing $u^i$, we observe that, by the  antisymmetry of $K$,
\begin{equation}
\label{antisym}
\int\! \rmd x\, \omega_{i,\eps}(x,t) \, \widetilde{u}^i(x,t) = \int\!\rmd x \int\! \rmd y\, \omega_{i,\eps}(x,t)\,\omega_{i,\eps}(y, t) \, K(x-y) = 0\,,
\end{equation}
and, as $(x-y) \cdot K(x-y) =0$, we get
\begin{gather*} 
\int\! \rmd x \, \omega_{i,\eps}(x,t) \, x \cdot \widetilde{u}^i(x,t) = \int\! \rmd x \int\! \rmd y \, \omega_{i,\eps}(x,t)\,\omega_{i,\eps}(y, t) \, x \cdot K(x-y) \\ = \int\! \rmd x \int\! \rmd y \, \omega_{i,\eps}(x,t)\,\omega_{i,\eps}(y, t) \, y \cdot K(x-y)\,,
\end{gather*} 
hence this integral is zero as well, by the antisymmetry of $K$. Using Proposition \ref{u e utilde} we then get,
\begin{gather*}
\left| \int\! \rmd x\, \omega_{i,\eps}(x,t)\, u^i(x,t) \cdot (x-B^i_\eps(t)) \, \right| \le \int\! \rmd x\, |\omega_{i,\eps}(x,t)| \, \big|u^i(x,t)-\widetilde{u}^i(x,t)\big| \, |x-B^i_\eps(t)| \\ \le \frac{C}{|\log\eps|^{\alpha -1}}\int\! \rmd x \, |\omega_{i,\eps}(x,t)| \, |x-B^i_\eps(t)| \le \: \frac{C}{|\log\eps|^{\alpha -1}} \,  I^i_\eps(t)^{1/2}\,,
\end{gather*}
where the Cauchy-Schwarz inequality has been used again in the last line. Gathering the above estimates we conclude that
\[
\dot{I}_\eps(t) \le 2C_F\,I^i_\eps(t) + \frac{C}{|\log\eps|^{\alpha -1}} I^i_\eps(t)^{1/2}\,.
\]
Defining $M_\eps(t):=I^i_\eps(t)^{1/2}$, by Gronwall's inequality and using the fact  that, by Eq.~\eqref{in_data2}, $I^i_\eps(0) \le 4\eps^2$, we get
\[
M_\eps(t) \le \left( 2\eps + \frac{C}{2 C_F|\log\eps|^{\alpha -1}} \right) \rme^{C_F t}  \le  \frac{C}{|\log\eps|^{\alpha -1}} \, \rme ^{C_F t}\,.
\]
As $T_{\eps,\beta,\eta} \le \eta\log|\log\eps|$, Eq.~\eqref{Iepsilon} follows from the bound above.
\end{proof}

In particular, Proposition \ref{lem:Ieps Beps} proves that if $\eta$ is chosen sufficiently small then $I^i_\eps(t)\to 0$ as $\eps\to 0$ up to time $T_{\eps,\beta,\eta}$. In the next corollary, we show that a similar result is true for the distances $|B^i_\eps(t) -z_i(t)|$.

\begin{corollary}
\label{cor:B-z}
For any $\eps$ small enough and $i=1,\ldots,N$,
\begin{equation} 
\label{B-z} 
|B^i_{\eps}(t) - z_i(t)| \le C \frac{|\log\eps|^{\eta(C_F+D)}}{|\log \eps|^{\alpha-1}} \qquad \forall\, t\in [0,T_{\eps,\beta,\eta}]\,,
\end{equation}
where $D := 2L\max_j |a_j|$ with $L$ the Lipschitz constant of $K(x)$ when $|x| \ge R_m/2$ ($R_m$ as in Eq.~\eqref{R_m}).
\end{corollary}

\begin{proof}
From the definition of $B^i_\eps(t)$, see Eq.~\eqref{cen_vort}, Eq.~\eqref{eqderivaterid} (with $f=x$), and Eq.~\eqref{pvmodel},
\begin{align*}
\dot B^i_{\eps}(t) - \dot z_i(t)  = & \: a_i^{-1} \int\! \rmd x\, \big(u^i(x, t) + F^i_\eps(x,t) \big) \, \omega_{i, \eps}(x,t) - \sum_{j \ne i} a_j K(z_i(t) - z_j(t))\, .
\end{align*}
We add and subtract appropriate terms, then by the splitting of $ F^i_\eps$ given in Lemma \ref{comesonoFeps} we get,
\begin{align}
\label{dotB}
\dot B^i_{\eps}(t) - \dot z_i(t) & = a_i^{-1} \int\! \rmd x\, u^i(x,t) \omega_{i, \eps}(x, t) + a_i^{-1}  \int\! \rmd x \, F^i_{\eps,2} (x, t) \omega_{i, \eps}(x, t) \nonumber \\ & \quad + a_i^{-1} \int\! \rmd x\,  [F^i_{\eps,1}(x,t) - F^i_{\eps,1}(B^i_{\eps}(t), t) ]\, \omega_{i, \eps} (x, t) \nonumber \\
& \quad + \sum_{j \neq i} \int\! \rmd y\,  [K(B^i_{\eps}(t)-y)- K(B^i_{\eps}(t)- B^j_{\eps}(t)) ] \, \omega_{j, \eps} (y, t) \nonumber \\ & \quad + \sum_{j \neq i} a_j [K(B^i_{\eps}(t)- B^j_{\eps}(t)) -K(B^i_{\eps}(t) -z_j(t) ) ] \nonumber \\
& \quad + \sum_{j \neq i} a_j [ K(B^i_{\eps}(t) -z_j(t) ) - K(z_i(t) - z_j(t)) ] \, .
\end{align}
The first term in the right-hand side is controlled by adding and subtracting $\widetilde u^i(x,t)$ inside the integral, using Proposition \ref{u e utilde}, and Eq.~\eqref{antisym}. For the other terms, we use Lemma \ref{comesonoFeps} and the Lipschitz property of $K$ when the absolute value of its argument is greater than a positive constant. Therefore, in view of Eqs.~\eqref{R_m}, \eqref{T_eps_beta}, and \eqref{Rm/2}, denoting by $L$ the Lipschitz constant of $K(x)$ when $|x| \ge R_m/2$,
\begin{align*}
|\dot B^i_{\eps}(t) - \dot z_i(t)| & \le\frac{|a_i|^{-1}C}{|\log \eps|^{\alpha-1}}  + |a_i|^{-1/2} C_F \, I^i_{\eps}(t)^{1/2} + \sum_{j \neq i} L |a_j|^{1/2}  I^j_{\eps}(t)^{1/2} \\ & \quad  + \sum_{j \neq i} L |a_j| |B^j_{\eps}(t) -z_j(t) | 
+ \sum_{j \neq i} L |a_j|  |B^i_{\eps}(t) -z_i(t) |\,.
\end{align*}
Setting 
\[
\Delta(t) := \sum_i |B^i_{\eps}(t) -z_i(t) |^2 \,,
\]
we have
\[
\begin{split}
\dot \Delta(t) & \le 2 \sum_i |B^i_{\eps}(t) -z_i(t)|\, |\dot B^i_{\eps}(t) - \dot z_i(t)| \\ & \le 4L\max_j |a_j| \Delta(t) + C\Bigg[\frac{1}{|\log \eps|^{\alpha-1}} + \bigg(\sum_j I^j_{\eps}(t)\bigg)^{1/2} \Bigg] \sqrt{\Delta(t)} \,.
\end{split}
\]
By integration of the previous inequality  and recalling $D := 2L\max_j |a_j|$,
\[
\sqrt{\Delta(t)} \le  \sqrt{\Delta(0)} \rme^{D t} + C \int_0^t\!\rmd s\,  \bigg(\sum_j I^j_{\eps}(t)\bigg)^{1/2} \rme^{D(t-s)} + \frac{C }{|\log \eps|^{\alpha-1}}(\rme^{Dt} -1)\,,
\]
with $\sqrt{\Delta(0)} \le \eps$ in view of Eq.~\eqref{in_data2}. Therefore, by the bound Eq.~\eqref{Iepsilon} and since $T_{\eps,\beta,\eta} \le \eta\log|\log\eps|$ we deduce that, for any $\eps$ small enough,
\begin{equation} 
\label{Deltat} 
\sqrt{\Delta(t)} \le C \frac{|\log\eps|^{\eta(C_F+D)}}{|\log \eps|^{\alpha-1}} \quad \forall\, t\in [0,T_{\eps,\beta,\eta}]\,.
\end{equation}
so Eq.~\eqref{B-z} is achieved.
\end{proof}

In the next lemma, we state a bound on the maximum relative velocity, with respect to the center of vorticity, of fluid particles near the boundary of $\Lambda_{i,\eps}(t)$.

\begin{lemma} 
\label{variazionemax}
Let
\begin{equation} 
\label{max_lambda} 
R_t := \max \{ |x-B^i_\eps(t)| \colon  x \in \Lambda_{i,\eps}(t) \}\,.
\end{equation}
Given $x_0 \in \Lambda_{i,\eps}(0)$, suppose at time $t \in [0,T_{\eps,\beta,\eta}]$ it happens that, for some $\rho>0$,
\begin{equation}
\label{erret}
R_t \ge 4\rho\,, \quad  R_t - \frac{\rho}{2} \le |\phi_t(x_0)- B^i_\eps(t)| \le R_t \,,
\end{equation}
Then at this time $t$ the following inequality holds,
\begin{align} 
\label{eq:variazonemax} 
\frac{\rmd}{\rmd t} |\phi_t(x_0)- B^i_\eps(t)| & \le \frac{16}7 |a_i| C_F \, |\phi_t(x_0)- B^i_\eps(t)|  + \frac{C I^i_\eps(t)}{ \rho^3} + \frac{C}{|\log\eps|^{\alpha -1}}  \nonumber \\ & \quad + \sqrt{\frac{3\, M \eps^{-2} m^i_\eps(R_t-\rho, t)}{\pi} }\,,
\end{align} 
where  the function $m^i_\eps$ is defined by
\begin{equation}
\label{funzemme}
m^i_\eps(R, t) := \int_{|y-B^i_\eps(t)|>R}\! \rmd y\, |\omega_{i,\eps}(y,t)| \quad \text{ for } R \in (0, +\infty)\,.
\end{equation}
\end{lemma}

\begin{proof}
Letting $x= \phi_t(x_0)$, we have
\begin{gather*}
\frac{\rmd}{\rmd t}  |\phi_t(x_0)- B^i_\eps(t)| = [u^i(x,t)+F^i_\eps(x,t)- \dot B^i_\eps(t)]  \cdot \frac{x-B^i_\eps(t)}{|x-B^i_\eps(t)|} \\ = \left[ \int\! \rmd y \, (F^i_\eps (x,t)-F^i_\eps (y,t)) \, \omega_{i,\eps}(y,t) \right] \cdot \frac{x-B^i_\eps(t)}{|x-B^i_\eps(t)|} \\ \quad + \left[ \int\! \rmd  y \, (u^i(x,t)-u^i(y,t)) \, \omega_{i,\eps}(y,t) \right] \cdot \frac{x-B^i_\eps(t)}{|x-B^i_\eps(t)|} \\ =: Q_F + Q_u .
\end{gather*}
The term involving $F^i_\eps$ is easily bounded using Lemma \ref{comesonoFeps} and Eq.~\eqref{norme*},
\begin{align}
\label{f1}
|Q_F| & \le \int\! \rmd y \, |F^i_\eps(x,t) - F^i_\eps(y,t)| \, |\omega_{i,\eps}(y,t)| \le C_F \int\! \rmd y \, |x-y| \, |\omega_{i,\eps}(y,t)|  \nonumber \\ & \quad + \frac{2C_F'\log|\log\eps|}{|\log \eps|^{\alpha}} \le 2|a_i|C_F R_t +\frac{2C_F'\log|\log\eps|}{|\log \eps|^{\alpha}} \nonumber \\ & \le\frac{16}7 |a_i| C_F \, |\phi_t(x_0)- B^i_\eps(t)| +\frac{2C_F'\log|\log\eps|}{|\log \eps|^{\alpha}}\,,
\end{align}
where in the last inequality we used that, by Eq.~\eqref{erret}, $R_t \le |\phi_t(x_0)- B^i_\eps(t)| + \frac 18 R_t$. For the term involving $u^i$, we use the identity $\int\!\rmd y\, \widetilde u^i(y,t)\,\omega_{i,\eps}(y,t) =0$ and Proposition \ref{u e utilde}, getting
\begin{align}
\label{f2}
|Q_u| & \le |u^i(x,t)- \widetilde{u}^i(x, t)| + \int\! \rmd y\, (u^i(y,t) - \widetilde{u}^i(y,t)) \, |\omega_{i,\eps}(y,t)| + |Q_u'| \nonumber \\ & \le  (1+|a_i|) \frac{C}{|\log\eps|^{\alpha -1}} + |Q_u'|\,,
\end{align}
where
\[
Q_u' =\frac{x-B^i_\eps(t)}{|x-B^i_\eps(t)|}  \cdot \widetilde{u}^i(x,t) = \frac{x-B^i_\eps(t)}{|x-B^i_\eps(t)|} \cdot \int\! \rmd y\, K(x-y) \, \omega_{i,\eps}(y,t)\,.
\]
To evaluate this (non trivial) term, we decompose the domain of integration into the two regions
\[
A_1 := \Sigma\big(B^i_\eps(t) \big| R_t - \rho \big)\,, \quad A_2:= \mathbb{R}^2 \setminus A_1\,.
\]
We call $H_1$ and $H_2$ the resultant integrals. We adapt for $H_1$ the proof of \cite[Lemma~2.5]{butmar}. Recalling Eq.~\eqref{nucleoK} and the notation $x^\perp = (x_2,-x_1)$ for $x = (x_1,x_2)$, after introducing the new variables $x'=x-B^i_\eps (t)$,  $y'=y-B^i_\eps (t)$, and using that $x'\cdot (x'-y')^\perp=-x'\cdot y'^\perp$, we get
\begin{equation}
\label{in H_11}
H_1 = \frac{1}{2\pi} \int_{|y'|\le R_t-\rho}\! \rmd y'\, \frac{x'\cdot y'^\perp}{|x'||x'-y'|^2}\, \omega_{i,\eps} (y'+B^i_\eps (t),t) \;.
\end{equation}
By  definition of  center of vorticity Eq.~\eqref{cen_vort}, $\int \rmd y'\,  y'^\perp\, \omega_{i,\eps} (y'+B^i_\eps (t),t) = 0$, so that
\begin{equation}
\label{in H_13*}
H_1  = H_1'-H_1''\;, 
\end{equation}
where
\begin{eqnarray*}
&& H_1' = \frac{1}{2\pi}  \int_{|y'|\le R_t-\rho}\! \rmd y'\, \frac {x'\cdot y'^\perp}{|x'|}\, \frac {y'\cdot (2x'-y')}{|x'-y'|^2 \ |x'|^2} \, \omega_{i,\eps} (y'+B^i_\eps (t),t)\;, \\ && H_1''= \frac{1}{2\pi} \int_{|y'|>R_t-\rho}\! \rmd y'\, \frac{x'\cdot y'^\perp}{|x'|^3}\, \omega_{i,\eps} 
(y'+B^i_\eps (t),t)\;.
\end{eqnarray*}
From Eq.~\eqref{erret} we have $|x'| \geq R_t -\frac{\rho}{2}$, and hence $|y'| \le R_t-\rho$ implies $|x'-y'|\ge \frac{\rho}{2}$ and $|2x'-y'|\le |x'-y'|+|x'| \le  |x'-y'| +R_t$, so that
\begin{equation*}
\begin{split}
|H_1'|&\le \frac{C}{\rho R_t^2 + \rho^2 R_t}  \int_{|y'|\le R_t-\rho} \! \rmd y'\, |y'|^2 \, |\omega_{i,\eps} (y'+B^i_\eps (t),t)| \le \frac{C I^i_\eps (t)}{\rho R_t^2 + \rho^2 R_t} \le \frac{C I^i_\eps (t)}{\rho^3}\,.
\end{split}
\end{equation*}
To bound $H_1''$ we note that, in view of Eq.~\eqref{max_lambda}, the integration is restricted to $|y'|\le R_t$, so that (using the lower bound in Eq.~\eqref{erret})
\begin{equation*}
|H_1''| \le \frac{C \, R_t}{\left(R_t-\rho\right)^2} \int_{|y'|> R_t-\rho}\! \rmd y'\, |\omega_{i,\eps} (y'+B^i_\eps (t),t)|\le \frac{C I^i_\eps (t)}{ R_t^3}\;,
\end{equation*}
where in the last bound we used the Chebyshev's inequality. By Eq.~\eqref{in H_13*} and the previous estimates we conclude that
\begin{equation}
\label{H_14}
|H_1| \le \frac{C I^i_\eps (t)}{ \rho^3}\;.
\end{equation}
We analyze now $H_2$. We first note that
\[ 
|H_2| \le \frac{1}{2\pi} \int_{|y-B^i_\eps(t)|>R_t-\rho}\! \rmd y \, \frac{1}{|x-y|} \, |\omega_{i,\eps}(y, t)| \,,
\]
so we can bound $H_2$ using again rearrangement, as in the proof of Proposition \ref{u e utilde}, i.e., we bound the integral taking a vorticity concentrated, as much as possible, around the singularity of $\frac{1}{|x-y|}$.
Therefore, the rearrangement is achieved replacing $\omega_{i,\eps}$ with the function equal to $3 M \eps^{-2}$ for $|x-y|<r$, and equal to $0$ for $|x-y| \ge r$, where $r$ is chosen such that $3 M \eps^{-2} \pi r^2= m^i_\eps(R_t-\rho,t)$ (which is the ``total mass'' of $\omega_{i,\eps}$ in the integration domain $A_2$). Then
\[
|H_2| \le \frac{3 M \eps^{-2}}{2 \pi} \int_{|z|<r} \! \frac{\rmd z}{z} = 3 M \eps^{-2} r = \sqrt{\frac{3\,M m^i_\eps(R_t-\rho, t)}{\pi \, \eps^2}}\,.
\]
From this last estimate and Eqs.~\eqref{f1}, \eqref{f2}, and \eqref{H_14}, Eq.~\eqref{eq:variazonemax} follows.
\end{proof}

We investigate now the behavior near to $0$ of the function
$m^i_\eps (\cdot, t)$ introduced in Eq.~\eqref{funzemme}.

\begin{lemma} 
\label{lem:mt}
Recalling Eqs.~\eqref{betabeta}, \eqref{Tebe}, and item (ii) in Lemma \ref{comesonoFeps}, assume
\begin{equation}
\label{eta1}
\eta <\frac{\alpha-1}{2 C_F}
\end{equation}
and fix
\begin{equation}
\label{beta_barra}
\beta<k<\frac{\alpha-1}{4}\,.
\end{equation}
Then, for each $\ell>0$ there exists $T_\ell>0$ such that, for any $\eps$ small enough,
\begin{equation}
\label{lemma_massa}
m^i_\eps\left(\frac{1}{|\log\eps|^k}, t\right) \le \frac{|a_i|}{|\log\eps|^\ell}
\quad \forall\, t\in  [0,T_\ell \wedge T_{\eps,\beta,\eta}]\,.
\end{equation} 
\end{lemma}

\begin{proof}
Given $R\ge 2h>0$, let $W_{R,h}(x)$, $x\in \mathbb{R}^2$, be a non-negative smooth function, depending only on $|x|$, such that
\begin{equation}
\label{W1}
W_{R,h}(x) = \begin{cases} 1 & \text{if $|x|\le R$}, \\ 0 & \text{if $|x|\ge R+h$}, \end{cases}
\end{equation}
and, for some $C_W>1$,
\begin{equation}
\label{W2}
|\nabla W_{R,h}(x)| < \frac{C_W}{h}\,,
\end{equation}
\begin{equation}
\label{W3}
|\nabla W_{R,h}(x)-\nabla W_{R,h}(x')| < \frac{C_W}{h^2}\,|x-x'|\,. 
\end{equation}

We define the quantity
\begin{equation}
\label{mass 1}
\mu_t(R,h) = \int\! \rmd x \, \big[1-W_{R,h}(x-B^i_\eps(t))\big]\, |\omega_{i,\eps} (x,t)|\,,
\end{equation}
which is a mollified version of $m^i_\eps$, satisfying
\begin{equation}
\label{2mass 3}
\mu_t(R,h) \le m^i_\eps(R, t) \le \mu_t(R-h,h)\,.
\end{equation}
In particular, it is enough to prove the claim with $\mu_t$ instead of $m^i_\eps$. The advantage is that the function $\mu_t$ is differentiable (with respect to $t$), with 
\begin{align*} 
\frac{\rmd}{\rmd t} \mu_t(R,h) & =  - \frac{a_i}{|a_i|} \int\! \rmd x \, \nabla W_{R,h}(x-B^i_\eps(t)) \cdot [ u^i(x,t)+F_\eps(x,t)-\dot{B}^i_\eps(t) ] \, \omega_{i,\eps}(x,t) \\ & = -H_3- H_4 - H_5\,,
\end{align*}
where, using Eq.~\eqref{antisym} and that $\dot{B}^i_\eps(t) = \int \rmd y \, \omega_{i,\eps}(y,t) [ F_\eps (y,t)+ u^i(y,t)]$,
\begin{align*}
H_3 & = \frac{a_i}{|a_i|} \int\! \rmd x \,\nabla W_{R, h}(x-B^i_\eps(t)) \cdot \widetilde{u}^i(x,t) \, \omega_{i,\eps}(x,t) \\
H_4 & = \frac{a_i}{|a_i|} \int\! \rmd x \, \nabla W_{R, h}(x- B^i_\eps(t)) \cdot \left[F_{\eps, 1}(x,t) - \int \rmd y\,  F_{\eps, 1}(y,t) \, \omega_{i,\eps}(y,t)\right] \, \omega_{i,\eps}(x,t) \\ H_5 & =  \frac{a_i}{|a_i|} \int\! \rmd x \, \nabla W_{R, h}(x-B^i_\eps(t))\\ & \;\;  \cdot \bigg[u^i(x,t)- \widetilde{u}^i(x,t) - \int\! \rmd y \,[u^i(y,t)- \widetilde{u}^i(y,t)] \, \omega_{i,\eps}(y,t) \bigg] \, \omega_{i,\eps}(x,t) \\ & \ + \frac{a_i}{|a_i|} \int\! \rmd x \,\nabla W_{R, h}(x-B^i_\eps(t)) \cdot \left[F_{\eps, 2}(x,t) - \int\! \rmd y \, F_{\eps, 2} (y, t) \, \omega_{i,\eps} (y, t) \right] \omega_{i,\eps}(x, t) \,.
\end{align*}
We immediately observe that, thanks to Proposition \ref{u e utilde}, Lemma \ref{comesonoFeps}, Eq.~\eqref{W2}, and to the fact that $\nabla W_{R, h}(z)$ is zero if $|z| \le R$,
\begin{equation}
\label{H5} 
|H_5| \le \frac{C_W}{h} \frac{C}{|\log\eps|^{\alpha-1}} \, m^i_\eps(R,t)\,.
\end{equation}
Along the same lines of \cite[Proposition~4.2]{butcavmar}  we find the following estimates, whose proofs are postponed in Appendix \ref{app:A}, for $H_3$ (corresponding to $A_1$ in \cite{butcavmar}),
\begin{equation} 
\label{H3} 
|H_3| \le C\left(\frac{1}{R^3 h} + \frac{1}{R^2 h^2}  \right) I^i_\eps(t) \, m^i_\eps(R,t) \, ,
\end{equation}
and for $H_4$ (corresponding to $A_3$ in \cite{butcavmar}),
\begin{equation} 
\label{H4} 
|H_4| \le C_W C_F  |a_i| \left( 1 + \frac{2R}{h} \right)  m^i_\eps(R,t) +  \frac{2 C_W C_F \, I^i_\eps(t)}{R^2 \, h} m^i_\eps(R,t) \, .
\end{equation}

Recalling Eq.~\eqref{Iepsilon}, from estimates Eqs.~\eqref{H5}, \eqref{H3},  and \eqref{H4}, we get
\begin{equation}
 \frac{\rmd}{\rmd t} \mu_t (R,h) \le A_\eps(R, h) m^i_\eps(R,t) \qquad \forall\, t\in  [0, T_{\eps,\beta,\eta}]\,,
\label{equ_mm}
\end{equation}
where
\begin{align}
\label{ARh}
A_\eps(R, h) & = \frac{C \, |\log\eps|^{2 \eta C_F }}{R^3 h |\log\eps|^{2(\alpha-1)}} +\frac{C |\log\eps|^{2 \eta C_F }}{R^2h^2 |\log\eps|^{2(\alpha-1)}} + \frac{C |\log\eps|^{2 \eta C_F }}{R^2h  |\log\eps|^{2(\alpha-1)}} \nonumber  \\ & \quad + \frac{C}{ h|\log\eps|^{\alpha-1}} + C  + 2|a_i| C_FC_W \frac Rh \,.
\end{align}
Therefore, by Eqs.~\eqref{2mass 3} and \eqref{equ_mm}, for any $\eps$ sufficiently small,
\begin{equation}
\label{da_iterare}
\mu_t (R,h) \le \mu_0 (R,h) + A_\eps(R, h) \int_0^t\, \rmd s  \, \mu_s (R-h, h) \qquad \forall\, t\in  [0, T_{\eps,\beta,\eta}]\,.
\end{equation}
We iterate the last inequality $n=\lfloor\log|\log\eps|\rfloor$  times (where $\lfloor a\rfloor$ denotes the integer part of the positive number $a$), from
\[
R_0 =\frac{1}{|\log\eps|^{k}}  \qquad \textnormal{to} \qquad R_n= R_0-n h = \frac{1}{2|\log\eps|^{k}} \,,
\]
with $k>0$ as in Eq.~\eqref{beta_barra}.  As a consequence,
\[
h= \frac{1}{2 n |\log\eps|^{k}} \sim \frac{\log|\log\eps|}{2|\log\eps|^k}
\]
and $R_j \in [R_0, R_n]$, $j=1,\dots, n$. Plugging these choices in Eq.~\eqref{ARh} we get that, for any $\eps$ small enough,
\begin{align}
\label{Arhj}
A_\eps(R_j, h) & \le  \frac{C|\log\eps|^{2 \eta C_F}} {|\log\eps|^{2(\alpha-1)}} \bigg( \frac{1}{R_j^3h} + \frac{1}{R_j^2h^2}\bigg) + \frac{C}{h|\log\eps|^{\alpha-1}} +  C  + 2|a_i| C_FC_W \frac{R_j}h \nonumber \\ & \le C \bigg(\frac{|\log\eps|^{4{k}+2 \eta C_F}(\log|\log\eps|)^2}{|\log\eps|^{2(\alpha -1)}} + \frac{\log|\log\eps|}{|\log\eps|^{\alpha-k-1}} + 1\bigg) \nonumber \\ & \quad + 4|a_i| C_FC_W  \log|\log\eps| \,.
\end{align}
In view of the assumptions Eqs.~\eqref{eta1} and \eqref{beta_barra}, Eq.~\eqref{Arhj} implies that, for any $\eps$ small enough, 
\begin{equation}
\label{CA}
A_\eps(R_j, h) \le 5 C_A \log|\log\eps |\,, \qquad C_A:= C_FC_W\max_i |a_i| \,.
\end{equation}
Therefore, for any $t\in  [0,T_{\eps,\beta,\eta}]$, it results 
\[
\begin{split}
\mu_t(R_0-h,h) & \le \mu_0(R_0-h,h) + \sum_{j=1}^{n-1} \mu_0(R_j,h) \frac{(5C_A\log |\log\eps|\, t)^j}{j!} \\ & \quad + \frac{(5C_A \log |\log\eps| )^{n}}{(n-1)!} \int_0^t\!{\textnormal{d}} s\,  (t-s)^{n-1}\mu_s(R_{n},h) \,.
\end{split}
\]
Since $\Lambda_\eps(0) \subset \Sigma(z_i|\eps)$, we can determine $\eps$ so small such that $\mu_0(R_j,h)=0$ for any $j=0,\ldots,n$, so that
\begin{align}
\label{mass 15'}
\mu_t(R_0-h,h) &\le \frac{\left(5C_A \log |\log\eps|\right)^{n}}{(n-1)!} \int_0^t\!{\textnormal{d}} s\,  (t-s)^{n-1}\mu_s(R_{n},h) \nonumber \\ & \le  \frac{\left(5C_A \log |\log\eps|\, t\right)^{n}}{n!}|a_i| \quad \forall\, t\in  [0, T_{\eps,\beta,\eta}]\,,
\end{align}
where the obvious estimate $\mu_s(R_{n},h) \le |a_i|$ has been used in the last inequality. In conclusion, using also Eq.~\eqref{2mass 3} and Stirling formula,
\[
m^i_\eps(R_0, t) \le \mu_t(R_0 -h,h) \le |a_i| (5\rme C_A t)^{\lfloor\log|\log\eps|\rfloor} \qquad \forall\, t\in  [0, T_{\eps,\beta,\eta}]\,,
\]
which implies Eq.~\eqref{lemma_massa} with, e.g., $T_\ell = (5C_A)^{-1} \rme^{-\ell-1}$.
\end{proof}

We need now a better result, in order to prove the compact support property of $\omega_{i,\eps} (\cdot, t)$, which is expressed in the following lemma.

\begin{lemma} 
\label{lem:mt2}
In the same assumptions of Lemma \ref{lem:mt}, for each $\ell>0$ there exists $T_\ell' >0$ such that, for any $\eps$ small enough,
\begin{equation}
\label{lemma_massa2}
m^i_\eps\left(\frac{3}{|\log\eps|^k}, t\right) \le |a_i| \eps^\ell \quad \forall\, t\in  [0,T_\ell' \wedge T_{\eps,\beta,\eta}]\,.
\end{equation}
\end{lemma}

\begin{proof}
The strategy of the proof is the same of the previous lemma, with the difference that we need now to take a larger number of iteration, that is $n=\lfloor|\log\eps|\rfloor$. With this choice, the first inequality in Eq.~\eqref{Arhj} is still true, but it implies, as now $h \sim |\log\eps|^{k+1}$,
\[
A_\eps(R_j, h)  \le C_A |\log\eps| + \frac{C|\log\eps|^{4k+2 \eta C_F+2}}{|\log\eps|^{2(\alpha -1)}}\,.
\]
The last term in the right-hand side gives some problems, at least in the range $1<\alpha\le 3/2$, since it diverges as $|\log\eps|^q$, with $q>1$, thus the iterative method does not work. On the other hand, this is due only to the (worst) term $1/(R_j^2h^2)$ appearing in Eq.~\eqref{Arhj}. This term, see the proof of Eq.~\eqref{H3} in Appendix \ref{app:A}, comes from the estimate Eq.~\eqref{badterm}, in which the quantity $\int_{|y'| \ge \frac{R}{2}}\!\rmd y'\, \tilde\omega_{i,\eps}(y',t) = m^i_\eps(R/2, t)$ is bounded from above by the moment of inertia. But, in view of Lemma \ref{lem:mt} with $\ell=\alpha+1$, if $R\ge 3|\log\eps|^{-k}$ then
\[
m^i_\eps(R/2, t) \le m^i_\eps\left(\frac{1}{|\log\eps|^k}, t\right) \le \frac{|a_i|}{|\log\eps|^{\alpha+1}}\quad \forall\, t\in  [0,T_{\alpha+1} \wedge T_{\eps,\beta,\eta}]\,,
\]
for any $\eps$ sufficiently small.

We then assume $t\in  [0,T_{\alpha+1} \wedge T_{\eps,\beta,\eta}]$ and iterate Eq.~\eqref{da_iterare} from
\[
R_0 =\frac{3}{|\log\eps|^k}  \qquad \textnormal{to} \qquad R_n= R_0-n h = \frac{5}{2|\log\eps|^k} \,,
\]
whence, consequently 
\[
h = \frac{1}{2n|\log\eps|^k} \sim \frac{1}{2|\log\eps|^{k+1}}\,.
\]
Therefore, $R_j \in [R_0, R_n]$,  $j=1, \dots, n$, so that, by using the above estimate for $m^i_\eps\left(R_j/2, t\right)$, the quantity  $A_\eps(R_j, h)$ is now bounded as
\begin{align*}
& A_\eps(R_j, h) \le \frac{C|\log\eps|^{2 \eta C_F}}{R_j^3h|\log\eps|^{2(\alpha-1)}} + \frac{C}{R_j^2h^2|\log\eps|^{\alpha+1}} + \frac{C}{h|\log\eps|^{\alpha-1}} + C +  2C_A \frac{R_j}h  \\ & \le C \bigg( \frac{|\log\eps|^{4{k}+2 \eta C_F+1}}{|\log\eps|^{2(\alpha -1)}} + |\log\eps|^{4k-\alpha+1} + |\log\eps|^{k-\alpha +2} + 1 \bigg) + 12 C_A |\log\eps|\,,
\end{align*}
whence, by assumptions Eqs.~\eqref{eta1} and \eqref{beta_barra}, $A_\eps(R_j, h) \le 13 C_A |\log\eps|$ for any $\eps$ small enough. Therefore, analogously to what made before,
\[
m^i_\eps(R_0, t) \le \mu_t(R_0 -h,h) \le  |a_i| (13 C_A t)^{\lfloor|\log\eps|\rfloor}\, ,
\]
which implies Eq.~\eqref{lemma_massa2} with $T_\ell' = [(13C_A)^{-1} \rme^{-\ell-1}] \wedge T_{\alpha+1}$.
\end{proof}

We are now ready to prove the main theorem.

\begin{proof}[Proof of Theorem \ref{teoEA}]
The proof is split into three steps.

\medskip
\noindent
\textit{Step 1}. Assume the parameters $\eta$ and $k$ are chosen according to Eqs.~\eqref{eta1} and \eqref{beta_barra}. We claim that, for any $i=1,\ldots,N$, and $\eps$ small enough,
\begin{equation}
\label{lamb}
\Lambda_{i,\eps}(t)\subseteq \Sigma \left(B^i_\eps (t) \, |\,  5\rho_k \right) \qquad \forall\, t\in  [0,T_3' \wedge T_{\eps,\beta,\eta}]\,,
\end{equation}
with $T_3'$ as in Lemma \ref{lem:mt2} for $\ell=3$, and
\[
\rho_k = \frac{1}{|\log\eps|^k}\,.
\]

To prove the claim, we start by noticing that, by Lemma \ref{variazionemax} with $\rho =\rho_k$, if at time $t\in [0,T_3' \wedge T_{\eps,\beta,\eta}]$ we have, for a certain $x_0\in \Lambda_{i,\eps}(0)$,
\begin{equation}
\label{corona}
R_t \ge 4\rho_k\,, \quad R_t - \frac 12 \rho_k \le |\phi_t(x_0)- B^i_{\eps}(t)|\le R_t\,,
\end{equation}
then the time derivative of $|\phi_t(x_0)-B^i_{\eps}(t)|$ is bounded by Eq.~\eqref{eq:variazonemax}, which implies, using also Eq.~\eqref{Iepsilon}, that, for any $\eps$ small enough,
\begin{align}
\label{max_sup}
\frac{\rmd}{\rmd t} |\phi_t(x_0)-B^i_\eps(t)|  & \le \frac{16}7 C_A |\phi_t(x_0)-B^i_\eps(t)| +  \frac{C  |\log\eps|^{3k+2\eta C_F}}{|\log\eps|^{2(\alpha-1)}} + \frac{C}{|\log\eps|^{\alpha-1}} \nonumber \\ & \quad + C \sqrt{\eps^{-2}m^i_\eps(R_t-\rho_k, t)}\,,
\end{align}
where we also used that $|a_i| C_F < C_A$ since $C_W>1$, see Eq.~\eqref{CA}. We now observe that by  Eqs.~\eqref{eta1} and \eqref{beta_barra} the second and third term in the right-hand side of Eq.~\eqref{max_sup} vanish faster than $\rho_k$ as $\eps\to 0$, and the same holds true for the last term by Lemma \ref{lem:mt2} with $\ell =3$ (which applies here as $R_t-\rho_k \ge 3|\log\eps|^{-k}$). On the other hand, if Eq.~\eqref{corona} is true then  $|\phi_t(x_0)-B^i_\eps(t)| \ge \frac 72 \rho_k$, so that Eq.~\eqref{max_sup} implies
\begin{equation}
\label{eq_phi}
\frac{\rmd}{\rmd t}  |\phi_t(x_0)-B^i_\eps(t)|   \le  3C_A |\phi_t(x_0)-B^i_\eps(t)|  \,,
\end{equation}
for any $\eps$ small enough.

We now conclude the proof of the claim. Suppose there is a first time $t_1 \in (0, T_{\eps,\beta,\eta}]$ such that $R_{t_1} = 5 \rho_k$ (otherwise the claim is verified), and define $t_0 = \sup\{t\in [0,t_1) \colon R_t < 4 \rho_k\}$. Clearly, $0 < t_0 < t_1$ (for any $\eps$ small enough) since $R_0 \le 3 \rho_k$; in particular $R_{t_0} = 4\rho_k$ and $R_t \ge 4\rho_k$ for any $t\in [t_0,t_1]$. Let us fix any fluid particle such that $|\phi_{t_1} (x_0)-B^i_\eps(t_1)| = R_{t_1}$ and let
\[
\sigma = \sup\Big\{t\in [t_0,t_1) \colon |\phi_t(x_0)-B^i_\eps(t)| \le \frac 92 \rho_k \Big\}\,.
\]
We notice that $\sigma \in (t_0, t_1)$ and Eq.~\eqref{corona} is verified for any $t\in [\sigma,t_1]$, so that Eq.~\eqref{eq_phi} holds for any $t\in [\sigma,t_1]$. Moreover, $|\phi_\sigma(x_0)-B^i_\eps(\sigma)| = \frac 92 \rho_k$. In conclusion, the function $g(t) := |\phi_t(x_0)-B^i_\eps(t)|$ satisfies 
\[
\dot g(t) \le 3C_A g(t) \quad \forall\, t\in  [\sigma,t_1]\,, \qquad g(\sigma) = \frac 92 \rho_k \,, \qquad g(t_1) = 5 \rho_k\,,
\]
which implies
\begin{equation}
\label{t1-tau}
t_1-\sigma \ge \frac{1}{3C_A} \log\frac{g(t_1)}{g(\sigma)} \ge \frac{1}{3C_A} \log\frac{10}9 > \frac{\rme^{-4}}{13C_A}  \ge T_3'\,.
\end{equation}
Therefore $t_1 \ge T_3' \wedge T_{\eps,\beta,\eta}$ and the claim Eq.~\eqref{lamb} follows.

\medskip
\noindent
\textit{Step 2}. We assume that $k$ is still chosen according to Eq.~\eqref{beta_barra}, while, regarding $\eta$, we require the stronger (with respect to Eq.~\eqref{eta1}) condition
\begin{equation}
\label{eta2}
\eta <\frac{\alpha-1}{2 (C_F+D)}\,,
\end{equation}
with $D$ as in Corollary \ref{cor:B-z}. With these choices, by Eq.~\eqref{B-z},
\[
|B^i_{\eps}(t) - z_i(t)| \le \frac{C}{|\log\eps|^{(\alpha-1)/2}} \qquad \forall\, t\in [0,T_{\eps,\beta,\eta}]\,,
\]
which, together with Eq.~\eqref{lamb} and since $k<(\alpha-1)/2$, implies that, for any $\eps$ small enough,
\[
\Lambda_{i,\eps}(t)\subseteq \Sigma \left(z_i (t) | 6 \rho_k \right) \qquad \forall\, t\in  [0,T_3' \wedge T_{\eps,\beta,\eta}]\,.
\]
Since $\beta<k$, in view of the definitions Eqs.~\eqref{T_eps_beta} and \eqref{Tebe}, a straightforward continuity argument show that $T_{\eps,\beta,\eta}$ must be necessarily larger than $T_3'$ for any $\eps$ small enough. In conclusion, for any $i=1,\ldots,N$, and $\eps$ small enough,
\[
\Lambda_{i,\eps}(t)\subseteq \Sigma \left(z_i (t) | 6 \rho_k \right) \qquad \forall\, t\in  [0,T_3']\,.
\]
This proves that as $\eps\to 0$ the vortex rings remain concentrated and their motion converges to the point vortex dynamics up to the ``positive but short'' time $T_3'$.

\medskip
\noindent
\textit{Step 3}. It remains to explain how to push forward the time of convergence up to the order of $\log|\log\eps|$. The rationale behind the strategy is to realize that, by iterating the previous argument, we end up with a sequence of time interval of convergence that, while shorter and shorter, are not summable.

With the first application of  Lemmata \ref{lem:mt} and \ref{lem:mt2}, combined with Step 2, we have obtained that the support for $\omega_{i,\eps}(\cdot, t)$ is contained in $\Sigma(B^i_\eps(t) | 5 \rho_k) \subseteq \Sigma \left(z_i (t) | 6 \rho_k \right)$ (with $\rho_k=|\log\eps|^{-k}$ and for any $\eps$ small enough) up to the time
\[
\mc T_1 = T_3' = \frac{\rme^{-4}}{13C_A} \wedge T_{\alpha+1} \ge \frac{1}{(3\cdot 2)\cdot 2+1} \tau\,,
\]
where $\tau := C_A^{-1} (\rme^{-4} \wedge \rme^{-\alpha-2})$. Using the data achieved at $\mc T_1$, we can repeat a second time both the iterative methods, which
are stated taking $(5+1)\rho_k$ in place of $\rho_k$ in Eq.~\eqref{lemma_massa}, and taking $(5+3) \rho_k$ in place of $3\rho_k$ in Eq.~\eqref{lemma_massa2}. The iteration is now performed from $(5+\frac12)\rho_k$ to $(5+1)\rho_k$ in Lemma \ref{lem:mt}, and from $(5+\frac52) \rho_k$ to $(5+3) \rho_k$ in Lemma \ref{lem:mt2}, i.e., in the last interval of width $\rho_k/2$ in both lemmata. This implies, arguing as in Steps 1 and 2, that the support of $\omega_{i,\eps} (\cdot, t)$ is contained in $\Sigma(B^i_\eps(t) | (5+5) \rho_k) \subseteq \Sigma \left(z_i (t) | (5+5+1) \rho_k \right)$, for any $\eps$ small enough, up to the time $\mc T(2) = \mc T_1 + \mc T_2$, with
\[
\mc T_2 = \frac{1}{33} \tau = \frac{1}{(8\cdot 2)\cdot 2 +1} \tau \,.
\]
With a third application of the same scheme, stated with $(10+1)\rho_k$ in place of $\rho_k$ in Eq.~\eqref{lemma_massa}, and with $(10+3)\rho_k$ in place of $3\rho_k$ in Eq.~\eqref{lemma_massa2}, always performing  the iteration in the last interval of width $\rho_k/2$ in both lemmata, we get that the support of $\omega_{i,\eps} (\cdot, t)$ is contained in $\Sigma(B^i_\eps(t) | (10+5) \rho_k) \subseteq \Sigma \left(z_i (t) | (10+5+1) \rho_k \right)$, for any $\eps$ small enough, up to the time $\mc T(3) = \mc T_1 + \mc T_2 + \mc T_3$, with
\[
\mc T_3 = \frac{1}{(13\cdot 2)\cdot 2 +1} \tau \,.
\]
We proceed further, considering at the $n^{\textnormal{th}}$ step $[5(n-1)+1]\rho_k$ in place of $\rho_k$ in Eq.~\eqref{lemma_massa}, and $[5(n-1)+3]\rho_k$ in place of $3\rho_k$ in Eq.~\eqref{lemma_massa2}, always performing  the iteration in the last interval of width $\rho_k/2$ in both lemmata. We get in this way that the support of $\omega_{i,\eps} (\cdot, t)$ is contained in $\Sigma(B^i_\eps(t) | 5n \rho_k) \subseteq \Sigma \left(z_i (t) | (5n+1) \rho_k \right)$ for any $\eps$ small enough up to the time\footnote{We remark this is achieved by using, in place of Eq.~\eqref{t1-tau}, the estimate 
\[
t_1-\sigma \ge \frac{1}{3C_A} \log\frac{g(t_1)}{g(\sigma)} \ge \frac{1}{3C_A} \log\frac{10j}{10j-1} > \frac{\rme^{-4}}{(20 j-7) C_A} \ge \frac{\tau}{[(5j-2)\cdot 2]\cdot 2+1}\,,
\]
where the third (strict) inequality is easily checked to be true for any integer $j\ge 1$.}
\begin{align*}
\mc T(n) = \sum_{j=1}^n \mc T_j & = \sum_{j=1}^n \frac{\tau}{[ (5j-2)\cdot 2]\cdot 2+1} \sim \frac{\tau}{20} \log n \quad \text{ as } n \to \infty\,.
\end{align*}

The above procedure is correct as long as $\Sigma(B^i_\eps(t) | 5j \rho_k) \subseteq \Sigma \left(z_i (t) | (5j+1) \rho_k \right)$ for any $j \le n$ and $\eps$ small enough (since this guarantees that $T_{\eps,\beta,\eta} > \mc T(j)$ for $j\le n$, see Step 2). This is clearly true if $n$ is arbitrary but fixed independent of $\eps$, provided $k$ and $\eta$ are chosen according to Eq.~\eqref{beta_barra} and \eqref{eta2}. But there is room to go further, by  choosing $n$ gently diverging as $\eps \to 0$.

Indeed, fix $0<k'<k-\beta$ and assume that $\eta$ satisfies the stronger condition
\[
\eta < \min\left\{\frac{\tau}{20}k', \frac{\alpha-1}{2(C_F+D)}\right\}.
\]
In this way, if we choose $n = n_\eps =|\log\eps|^{k'}$ then $\mc T(n_\eps) > \eta\log |\log\eps|$, eventually as $\eps \to 0$. Moreover, the above scheme can be successfully iterated for any $j \le n_\eps$. Indeed, by its application,
\[
R_t \le 5 j \rho_k \le 5 \frac{|\log\eps|^{k'}}{|\log\eps|^k} \qquad \forall\, t\in [0,\mc T(j) \wedge T_{\eps,\beta,\eta}] \,,
\]
which gives the inclusion $\Lambda_{i,\eps}(t) \subseteq \Sigma \left(z_i (t) | (5j+1) \rho_k \right)$ by Eq.~\eqref{B-z}. Since $\beta < k-k'$, this in turn implies $T_{\eps,\beta} \ge \mc T(j)$ for any $j\le n_\eps$ and any $\eps$ small enough. Theorem \ref{teoEA} thus follows with $\zeta=\eta$.
\end{proof}

\appendix

\section{Proofs of Eqs.~\eqref{H3} and \eqref{H4}}
\label{app:A}

Without loss of generality, we assume $\omega_{i,\eps}(x,t) \ge 0$. We write,
\[
\begin{split}
H_3 & = \int\! \rmd x\, \nabla W_{R,h}(x-B^i_\eps(t)) \cdot \int\! \rmd y \, K(x-y)\, \omega_{i,\eps}(y,t)\, \omega_{i,\eps}(x,t)  \\ & = \frac 12 \int\! \rmd x \! \int\! \rmd y\, [\nabla W_{R,h}(x-B^i_\eps(t)) - \nabla W_{R,h}(y-B^i_\eps(t))] \\ & \qquad  \cdot K(x-y) \, \omega_{i,\eps}(x,t)\,  \omega_{i,\eps}(y,t) \\ 
H_4 & =   \int\! \rmd x\, \nabla W_{R,h}(x-B^i_\eps(t)) \cdot \int\! \rmd y \,[F_{\eps, 1}(x,t)-F_{\eps, 1}(y,t)]\, \omega_{i,\eps}(y,t)\, \omega_{i,\eps}(x,t)\,, 
\end{split}
\]
where the second expression of $H_3$ is due to the antisymmetry of $K$.

Concerning $H_3$, we introduce the new variables $x'=x-B^i_\eps(t)$, $y'=y-B^i_\eps(t)$, define $\tilde\omega_{i,\eps}(z,t) := \omega_{i,\eps}(z+B^i_\eps(t),t)$, and let
\[
f(x',y') = \frac 12 \tilde\omega_{i,\eps}(x',t)\, \tilde\omega_{i,\eps}(y',t) \, [\nabla W_{R,h}(x')-\nabla W_{R,h}(y')] \cdot K(x'-y') \,,
\]
so that $H_3 = \int\!\rmd x' \! \int\!\rmd y'\,f(x',y')$. We observe that $f(x',y')$ is a symmetric function of $x'$ and $y'$ and that, by Eq.~\eqref{W1}, a necessary condition to be different from zero is if either $|x'|\ge R$ or $|y'|\ge R$. Therefore,
\begin{equation*}
\begin{split}
H_3  &= \bigg[ \int_{|x'| > R}\!\rmd x' \! \int\!\rmd y' + \int\!\rmd x' \! \int_{|y'| > R}\!\rmd y' -  \int_{|x'| > R}\!\rmd x' \! \int_{|y'| > R}\!\rmd y'\bigg]f(x',y') \\ & = 2 \int_{|x'| > R}\!\rmd x' \! \int\!\rmd y'\,f(x',y')  -  \int_{|x'| > R}\!\rmd x' \! \int_{|y'| > R}\!\rmd y'\,f(x',y') \\ & = H_3' + H_3'' + H_3'''\,,
\end{split}
\end{equation*}
with 
\begin{equation*}
\begin{split}
H_3' & = 2 \int_{|x'| > R}\!\rmd x' \! \int_{|y'| \le \frac{R}{2}}\!\rmd y'\,f(x',y') \,, \quad H_3'' = 2 \int_{|x'| > R}\!\rmd x' \! \int_{|y'| > \frac{R}{2}}\!\rmd y'\,f(x',y')\,, \\ H_3''' & = -  \int_{|x'| > R}\!\rmd x' \! \int_{|y'| > R}\!\rmd y'\,f(x',y')\,.
\end{split}
\end{equation*}
By the assumptions on $W_{R,h}$, we have $\nabla W_{R,h}(z) = \eta_h(|z|) z/|z|$ with $\eta_h(|z|) =0$ for $|z| \le R$. In particular, $\nabla W_{R,h}(y') = 0$ for $|y'| \le R/2$, hence
\[
H_3' =  \int_{|x'| > R}\!\rmd x' \, \tilde\omega_{i,\eps}(x',t) \eta_h(|x'|) \,\frac{x'}{|x'|} \cdot  \int_{|y'| \le \frac{R}{2}}\!\rmd y'\, K(x'-y') \, \tilde\omega_{i,\eps}(y',t)\,.
\]
In view of  Eq.~\eqref{W2}, $|\eta_h(|z|)| \le C_W/h$, so that 
\begin{equation}
\label{a1'}
|H_3'| \le \frac{C_W}{h} m^i_\eps(R,t) \sup_{|x'| > R} |\widetilde H(x')|\,,
\end{equation}
with
\[
\widetilde H(x') = \frac{x'}{|x'|}\cdot  \int_{|y'| \le \frac{R}{2}}\!\rmd y'\, K(x'-y') \, \tilde\omega_{i,\eps}(y',t) \,.
\]
Now, recalling Eq.~\eqref{nucleoK} and using that $x'\cdot (x'-y')^\perp=-x'\cdot y'^\perp$, we get,
\[
\widetilde H(x') = \frac{1}{2\pi} \int_{|y'|\le \frac{R}{2}}\! \rmd y'\, \frac{x'\cdot y'^\perp}{|x'||x'-y'|^2}\,  \tilde\omega_{i,\eps}(y',t) \,.
\]
By Eq.~\eqref{cen_vort}, $\int\! \rmd y'\,  y'^\perp\,  \tilde\omega_{i,\eps}(y',t) = 0$, so that
\begin{equation}
\label{in H_13}
\widetilde H(x')  = \widetilde H'(x')-\widetilde H''(x')\,, 
\end{equation}
where
\begin{eqnarray*}
&& \widetilde H'(x') = \frac{1}{2\pi}  \int_{|y'|\le \frac{R}{2}}\! \rmd y'\, \frac {x'\cdot y'^\perp}{|x'|}\, \frac {y'\cdot (2x'-y')}{|x'-y'|^2 \ |x'|^2} \,  \tilde\omega_{i,\eps}(y',t) \,, \\ && \widetilde H''(x')= \frac{1}{2\pi} \int_{|y'|> \frac{R}{2}}\! \rmd y'\, \frac{x'\cdot y'^\perp}{|x'|^3}\,  \tilde\omega_{i,\eps}(y',t) \,.
\end{eqnarray*}
We notice that if $|x'| > R$ then $|y'| \le \frac{R}{2}$ implies $|x'-y'|\ge \frac{R}{2}$ and $|2x'-y'|\le |x'-y'|+|x'|$. Therefore, for any $|x'| > R$,
\[
|\widetilde H'(x')| \le \frac 1\pi \bigg[\frac{1}{|x'|^2 R} + \frac{2}{|x'|R^2} \bigg]  \int_{|y'|\le \frac{R}{2}} \! \rmd y'\, |y'|^2 \,  \tilde\omega_{i,\eps}(y',t) \le \frac{3 I^i_\eps(t)}{\pi R^3}\,.
\]
To bound $\widetilde H''(x')$, by Chebyshev's inequality, for any $|x'| > R$ we have,
\[
|\widetilde H''(x')| \le \frac{1}{2\pi |x'|^2} \int_{|y'|> \frac{R}{2}}\! \rmd y'\, |y'| \tilde\omega_{i,\eps}(y',t) \le \frac{I^i_\eps(t)}{ \pi R^3} \,.
\]
From Eqs.~\eqref{a1'}, \eqref{in H_13}, and the previous estimates, we conclude that
\begin{equation}
\label{H_14b}
|H_3'| \le \frac{4C_W I^i_\eps(t)}{\pi R^3  h} m^i_\eps(R,t)\,.
\end{equation}
Now, by Eq.~\eqref{W3} and then applying the Chebyshev's inequality,
\begin{align}
\label{badterm}
|H_3''| + |H_3'''| & \le \frac{C_W}{\pi h^2} \int_{|x'| \ge R}\!\rmd x' \! \int_{|y'| \ge \frac{R}{2}}\!\rmd y'\,\tilde\omega_{i,\eps}(y',t) \,  \tilde\omega_{i,\eps}(x',t) \nonumber \\ & = \frac{C_W}{\pi h^2}m^i_\eps(R,t)   \int_{|y'| \ge \frac{R}{2}}\!\rmd y'\, \tilde\omega_{i,\eps}(y',t)  \le \frac{4C_W I^i_\eps(t)}{\pi R^2h^2} m^i_\eps(R,t)\,.
\end{align}
Gathering together the previous estimates, Eq.~\eqref{H3} follows.

Concerning $H_4$, we observe that by Eq.~\eqref{W1} the integrand is different from zero only if $R\le |x-B^i_\eps(t)|\le R+h$. Therefore, by Lemma \ref{comesonoFeps} and Eq.~\eqref{W2} we have, using again the variables $x'=x-B^i_\eps(t)$, $y'=y-B^i_\eps(t)$,
\[
\begin{split}
|H_4| & \le \frac{2C_W C_F}{h} \int_{|x'|\ge R}\! \rmd x'  \tilde\omega_{i,\eps}(x',t)  \int_{|y'|> R}\! \rmd y'\, \tilde\omega_{i,\eps}(y',t) \\ &\quad + \frac{C_W C_F }{h} \int_{R \le |x'|\le R+h}\! \rmd x'  \tilde\omega_{i,\eps}(x',t) \int_{|y'| \le R}\! \rmd y'\,|x'- y'| \,  \tilde\omega_{i,\eps}(y',t) \,.
\end{split}
\]
This bound implies Eq.~\eqref{H4} by using the Chebyshev's inequality in the first integral and that $|x'-y'| \le 2R+h$ in the domain of integration of the last integral.

\bigskip
\noindent {\textbf{Acknowledgments}}.
Work performed under the auspices of GNFM-INDAM and the Italian Ministry of the University (MIUR).

\end{document}